\documentclass{amsart}
\usepackage{amsfonts}
\usepackage{amsmath,amscd}
\usepackage{amssymb}
\usepackage{amsthm}
\usepackage{newlfont}

\long\def\alert#1{\parindent2em\smallskip\hbox to\hsize%
{\hskip\parindent\vrule%
\vbox{\advance\hsize-2\parindent\hrule\smallskip\parindent.4\parindent%
\narrower\noindent#1\smallskip\hrule}\vrule\hfill}\smallskip\parindent0pt}
 \newtheorem{thm}{Theorem}[section]
 \newtheorem{cor}[thm]{Corollary}
 \newtheorem{lem}[thm]{Lemma}
 \newtheorem{prop}[thm]{Proposition}
 \theoremstyle{definition}
 
 \theoremstyle{remark}

 \numberwithin{equation}{section}
\begin{document}

\title[certain functors of  some  Lie algebras of class two]
{ The capability and certain functors of some nilpotent Lie algebras of class two}

%\author[F. Johari]{Farangis Johari}
%\author[M. Parvizi]{Mohsen Parvizi}
%\author[P. Niroomand]{Peyman Niroomand}
%\address{Department of Pure Mathematics\\
%Ferdowsi University of Mashhad, Mashhad, Iran}
%\email{farangisjohary@yahoo.com}
%
%\address{Department of Pure Mathematics\\
%Ferdowsi University of Mashhad, Mashhad, Iran}
%\email{parvizi@math.um.ac.ir}
%
%\address{School of Mathematics and Computer Science\\
%Damghan University, Damghan, Iran}
%\email{niroomand@du.ac.ir}
%
%
%
%
%%\thanks{\textit{Mathematics Subject Classification 2010.} Primary 20F99 Secondary 20F14.}
%
%
%\keywords{Tensor square, exterior square, capability, Schur multiplier, $p$-groups, relative Schur multiplier, locally finite groups}
%
%\date{\today}
\author[F. Johari]{Farangis Johari} 
\author[P. Niroomand]{Peyman Niroomand}

\address{ Departamento de Matem\'{a}tica, Instituto de Ci\^{e}ncias Exatas, Universidade Federal de Minas Gerais, Av. Ant\^{o}nio Carlos 6627, Belo Horizonte, MG, Brazil.}
\email{farangisjohari@ufmg.br,farangisjohari85@gmail.com}

\address{School of Mathematics and Computer Science\\
Damghan University, Damghan, Iran}
\email{niroomand@du.ac.ir, p$\_$niroomand@yahoo.com}

\thanks{\textit{Mathematics Subject Classification 2010.} Primary 17B30; Secondary 17B05, 17B99.}

\keywords{Nilpotent Lie algebra, Capability, Schur multiplier, Tensor square, Exterior square}

\date{\today}
%\dedicatory{}
\begin{abstract}
Recently, the authors obtained the Schur multiplier, the non-abelian tensor square and the non-abelian exterior square of $d$-generator generalized Heisenberg Lie algebras of rank $ \frac{1}{2}d(d-1).$ Here, we intend to obtain the same results for $d$-generator generalized Heisenberg Lie algebras of rank $  t$ when  $ \frac{1}{2}d(d-1)-3 \leq t\leq \frac{1}{2}d(d-1)-1.$ Then, as a result, we give similar consequences for a nilpotent Lie algebra $ L $  of class two when $ \dim (L/Z(L))=d,$ $ \dim L^2=t $  such that $ \frac{1}{2}d(d-1)-3 \leq t\leq \frac{1}{2}d(d-1)-1.$ 
\end{abstract}

%%% ----------------------------------------------------------------------
\maketitle
%%% ----------------------------------------------------------------------

\section{Motivation and Preliminary results}
The Schur multiplier of a group $ G $ comes from the work of Schur in 1904. It is well-known that for a group $ G $ with a free representation $ F/R, $ we have \[ \mathcal{M}(G)\cong (R\cap F')/[R,F]. \] There are many papers devoted to obtain the structure of 
the Schur multiplier of  groups, when  the structure of  groups is in hand. For instance, the Schur multiplier of abelian groups, extra-special $p$-groups, and generalized extra-special   $p$-groups are obtained in \cite{kar, jafa}. Rai in \cite{Ra} succeed to obtain the Schur multiplier of  special   $p$-groups  minimality generated by $ d\geq 3 $ elements with   derived subgroup of order $ p^{\frac{1}{2}d(d-1)}. $ It is well-known that the results of the class of  finite $p$-groups have analogues on the class of finite-dimensional  nilpotent  Lie algebras. This motivation allows us to write this paper. Let $L$ be a Lie algebra and $L\cong  F/R$ for a free Lie algebra $F. $ It is well-known that the Schur multiplier $\mathcal{M}(L) $ of $L$ is equal to $(R \cap F^2)/[R, F].$ Recall that a  Lie algebra $L$ is called capable if and only if $L\cong E/Z(E)$ for some  Lie algebra $E.$
By the results of 
\cite{ba1,cer,Mon,ni60}, the Schur multiplier and the capability of nilpotent Lie algebras with derived subalgebra of dimension  $ t $ such that $0\leq t\leq 2$ are obtained.  Recall that 
a Lie algebra $H$ is called a generalized Heisenberg Lie algebra of rank $n$ if $H^2 = Z(H)$ and $\dim H^2 = n.$ The capability and the structures of the Schur multiplier, the non-abelian tensor square, and the exterior square of $d$-generator generalized Heisenberg Lie algebras of rank $ \frac{1}{2}d(d-1) $ when $ d\geq 3 $ are given in \cite{ni600}. The motivation of this paper  is to obtain  the Schur multiplier, the non-abelian tensor square and the exterior square of $d$-generator generalized Heisenberg Lie algebras of rank $  t$ such that  $ \frac{1}{2}d(d-1)-3 \leq t\leq \frac{1}{2}d(d-1)-1.$  Moreover, we give the same results for a nilpotent Lie algebra $L$ of class two with $ \dim (L/Z(L))=d$ and $ \dim L^2= t$ such that $\frac{1}{2}d(d-1)-3\leq t\leq \frac{1}{2}d(d-1)-1.$ Finally, we show that all
of $d$-generator generalized Heisenberg Lie algebras of rank  $ \frac{1}{2}d(d-1)-1$ and $\frac{1}{2}d(d-1)-2$ are capable and also we show that all of nilpotent Lie algebras  of class two with central factor of dimension $ d$ and derived subalgebra of dimension  $ t $ such that $\frac{1}{2}d(d-1)-2\leq t\leq \frac{1}{2}d(d-1)-1$ are capable.\\ Throughout the paper $A(n)$ is used to denote an abelian Lie algebra of dimension $n.$ 
For a given  Lie algebra $ L $ and $ x\in L, $ let $ Y=L/\big{(}L^2+Z(L) \big{)}$ and $ \overline{x}=x+L^2+Z(L). $ Then
\begin{thm}\cite[Proposition 2.3]{nif}\label{lkk}
Let $ L $ be a finite-dimensional  nilpotent Lie algebra of class two. Then the map \[\Psi_2:  Y\otimes   Y\otimes  Y  \rightarrow L^2\otimes L/L^2\] given by
$\overline{x}\otimes \overline{y}\otimes \overline{z}\mapsto  ([x,y]\otimes z+L^2)+
([z,x]\otimes y+L^2)+([y,z]\otimes x+L^2)
$
is a  homomorphism. Moreover, if any two elements of the set $\{ x,y,z\} $ are linearly dependent, then $\Psi_2(\overline{x} \otimes \overline{y} \otimes \overline{z})$ is identity.
\end{thm}

 Similar to the result of Blackburn for the group theory case (see \cite{bl}), we begin with the following result for  Lie algebras. 
Let  $ L $ be a  finite-dimensional nilpotent  Lie algebra  of class two with a  free presentation $0\rightarrow R \rightarrow F \overset{\pi}{\rightarrow} L \rightarrow 0$ for a free Lie algebra $ F$ and an ideal $ R $ in $ F. $ 
 If $ x\in L, $ then $ \tilde{x}\in F $ denotes a fixed pre-image of $ x $ under $ \pi. $
 As was shown in \cite[Proposition 4.1]{sa},,  the following map is a homomorphism
 \begin{align}\label{eq}
\beta :L^ 2 \otimes L /L^2 &\rightarrow  \mathcal{M}(L)=[\tilde{x }, \tilde{z}] + [R, F]\nonumber\\
x \otimes (z + L^2 ) 
&\mapsto [\tilde{x }, \tilde{z}] + [R, F].
\end{align}
Clearly, $\mathrm{Im}(\beta)=F^3+[R,F]/[R,F]. $ 
 
 \begin{lem}(\cite[Theorem 2.1]{nif} and \cite[Proposition 4.1]{sa})\label{1}
For a given finite-dimensional nilpotent  Lie algebra $L $ of class two,  the sequence  \[0 \rightarrow \ker \beta \rightarrow L^2 \otimes L/L^2 \overset{\beta}{\rightarrow} \mathcal{M}(L) \rightarrow \mathcal{M}(L/L^2 ) \rightarrow L^2 \rightarrow 0\] is exact.  Moreover, 
 \[K = \langle [x, y] \otimes  z + L^ 2 + [z, x] \otimes y + L^ 2 + [y, z] \otimes x + L^ 2 |x, y, z \in L\rangle \subseteq  \ker \beta.\]
\end{lem}
Recall that a pair of Lie algebras $(E, M)$ is said to be a defining pair for a Lie algebra $L$ if
$M \subseteq  Z(E) \cap  E^2$ and $E/M \cong  L.$ If $L $ is finite-dimensional, then the dimension
of $ E  $ is bounded. If $(E, M)$ is a defining pair for $L,$ then a $E$ of maximal dimension
is called a cover for $L.$ Moreover, from \cite{ba} and \cite{Mon} in this case $M\cong  \mathcal{M}(L). $ From \cite{Mon}, all covers of a finite-dimensional Lie algebra $ L $ are isomorphic. \\
Let  $L $ be a finite-dimensional nilpotent  Lie algebra of class two such that \[ M^*=(L/L^2\wedge L/L^2) \oplus \big{(}(L^2\otimes L/L^2)/K\big{)}~ \text{ and}~ N\cong\big{(}(L^2\otimes L/L^2)/K\big{)}, \] where $ K $ defined in Lemma \ref{1}. Then $ M^*/N\cong  (L/L^2\wedge L/L^2).$ It is clear that   the map $\rho: L/L^2\wedge L/L^2\rightarrow L^2 $ given by $ x+L^2\wedge y+L^2\mapsto [x,y] $
is an epimorphism. Let $M$ be a subalgebra of $M^*$ containing $ N $
for which $M/N\cong \ker \rho.  $ Since $ M^* $  is abelian, we have $ M=N\oplus \ker \rho.$ Clearly, \[  \dim M/N=\dim (L/L^2\wedge L/L^2)-\dim L^2.\] In the following result, we will  construct a Lie algebra   $ L^* $  containing $M^*$ such that  $M^*\subseteq Z(L^*), $ 
 $ L^*/ M^* \cong L$ and  
 $M^* \cap (L^*)^2=M.$  Let $0\rightarrow R \rightarrow F \overset{\pi}{\rightarrow} L \rightarrow 0$ be a  free presentation for $ L^* $   for a free Lie algebra $ F$ and an ideal $ R $ in $ F. $ 
 Assume that $ M^*\cong S/R $ for some ideal $ S $ of $ F. $ Then $ L^*/M^*\cong F/S $ and $ M^* \cap (L^*) ^2\cong (F^2\cap S)+R/R\cong (F^2\cap S)/(F^2\cap R). $ Since $ M^* $ is a central ideal, we get $ [S,F]\subseteq R. $ Now, using the proof of \cite[Proposition 4.1]{sa}, there is the following  epimorphism
 \begin{align}\label{e3}
\alpha: \mathcal{M}(L^{*}/M^{*})=(F^2\cap S)/[S,F]&\rightarrow  M^* \cap (L^*) ^2\cong (F^2\cap S)/(F^2\cap R)\nonumber\\
  f+ [S,F] &\mapsto f+(F^2\cap R).
 \end{align}
With the above notations and assumptions, we have
\begin{thm}\label{2}
$\mathcal{M}(L)\cong M $ and $ \ker \beta=K. $
\end{thm}
\begin{proof}
Using Lemma \ref{1}, we have  \begin{align*}
 &\dim  \mathcal{M}(L)=\dim \mathcal{M}(L/L^2)-\dim L^2+\dim (F^3+[R,F]/[R,F])=\\&\dim \mathcal{M}(L/L^2)-\dim L^2+\dim \big{(}(L^2\otimes L/L^2)/\ker \beta \big{) }\\&\leq\dim \mathcal{M}(L/L^2)-\dim L^2+\dim \big{(}(L^2\otimes L/L^2)/K\big{)}.
 \end{align*}
 Since $ \dim \mathcal{M}(L/L^2)=\dim (L/L^2\wedge L/L^2),  $ we have
  \begin{align*} \dim  \mathcal{M}(L)&\leq    \dim (L/L^2\wedge L/L^2)-\dim L^2+\dim \big{(}(L^2\otimes L/L^2)/K\big{)}\\&=\dim (M/N)+\dim N=\dim M.\end{align*} Hence, $ \dim  \mathcal{M}(L)\leq \dim M.$
 To prove $ \mathcal{M}(L)\cong M,  $ it is  sufficient to construct a Lie algebra $ L^* $  containing $M^*$ such that  $M^*\subseteq Z(L^*), $ 
 $ L^*/ M^* \cong L$ and  
 $M^* \cap (L^*)^2=M,$ since in this case by using (\ref{e3}), we can consider the epimorphism $ \alpha: \mathcal{M}(L)\rightarrow M.$ Now since $ \dim  \mathcal{M}(L)\leq \dim M, $  $ \alpha $ is an isomorphism. Hence, $ \mathcal{M}(L)\cong M. $ In particular, since  
  \begin{align*}
  &\dim  \mathcal{M}(L)=\dim (L/L^2\wedge L/L^2)-\dim L^2+\dim \big{(}(L^2\otimes L/L^2)/\ker \beta \big{)}\leq  \\&  \dim (L/L^2\wedge L/L^2)-\dim L^2+\dim \big{(}(L^2\otimes L/L^2)/K\big{)}\\&
  =\dim M/N+\dim N=\dim M=\dim \mathcal{M}(L)\end{align*} and $ K\subseteq \ker \beta, $ we have $  \dim \big{(}(L^2\otimes L/L^2)/K\big{)}=\dim \big{(}(L^2\otimes L/L^2)/\ker \beta\big{)}$ and so $K= \ker \beta, $ as required.
 \\
 We are going to construct  a Lie algebra $ L^* $   such that  
 $ L^*/ M^* \cong L.$
Choose a basis $ \{x_1+L^2,\ldots,x_m+L^2\}$ for $ L/L^2 $ and a basis $ \{y_1,\ldots,y_r\}$   for $ L^2. $ 
Since  $L $ is of class two, $ L $ has the following presentation 
\[
L=\langle x_i,y_s|[x_i,x_j]=\sum_{d=1}^{r} \alpha_{d_{ij}}y_d,  [y_s,x_i]=[y_{s'},y_s]=0, \]\[,\alpha_{d_{ij}}\in \mathbb{F},1\leq i,j\leq m,1\leq s,s'\leq r\rangle.\]
Suppose that
  $ a_{si} =(y_s\otimes \big{(}x_i+L^2)\big{)}+K$ and $ e_{ij}=x_{i}+L^2\wedge x_{j}+L^2\in L/ L^2\wedge L/L^2$ for all $ i,j,s $ such that  $1 \leq s\leq r$ and $1 \leq i,j\leq m.$ Then 
  $ N=\langle  a_{si} |1 \leq s\leq r,1 \leq i\leq m\rangle $ and $ M^*=\langle e_{ij}|1 \leq i,j\leq m  \rangle \oplus N. $  Let $  \overline{x}_i$ and $\overline{y}_s$ be  pre-images of elements $ x_i$ and $y_s$ in $ L. $ Assume that $  L^*$ is a Lie algebra generated by $M^*\cup \{\overline{x}_i,\overline{y}_s|1 \leq s\leq r,1 \leq i\leq m\}$ with the following relations
\begin{align*}
&[\overline{x_i},\overline{x_j}]=\sum_{d=1}^{r} \alpha_{d_{ij}}\overline{y_d}+e_{ij}, [\overline{y_s},\overline{x_i}]= a_{si},\\ &[e_{ij},\overline{x_{i'}}]=[e_{ij},\overline{y_s}]=
[ a_{si},\overline{x_{i'}}]=[ a_{si},\overline{y_{t'}}]=[\overline{y_{t}},\overline{y_{t'}}]=0
\end{align*} for all $ i,j,s,i',t,t' $ such that $ 1 \leq s,t,t'\leq r$ and $1 \leq i,j,i'\leq m. $ One can easily check that 
 $ (L^*)^2 =N\oplus \langle \overset{r}{\underset{d=1}{\sum}} \alpha_{d_{ij}}\overline{y_d}+e_{ij}| 1 \leq i,j\leq m\rangle,$ $M^*\subseteq Z(L^*), $ and $ L^*/M^*\cong L.$ In particular,  $\overline{x_i}+M^*$ and $ \overline{y_s}+M^*$ correspond to  $ x_i $ and    $y_s,$ respectively. Since $ N\subseteq M^*, $ we have $M^* \cap (L^*)^2=N\oplus (M^*\cap   \langle \overset{r}{\underset{d=1}{\sum}} \alpha_{d_{ij}}\overline{y_d}+e_{ij}| 1 \leq i,j\leq m\rangle). $
 We claim that $ M^*\cap   \langle \overset{r}{\underset{d=1}{\sum}} \alpha_{d_{ij}}\overline{y_d}+e_{ij}| 1 \leq i,j\leq m\rangle=\ker \rho. $ If $\underset{1\leq i,j\leq m}{\sum} \omega_{ij} e_{ij}\in \ker \rho, $  where $  \omega_{ij}$ is a scalar, then $ \underset{1\leq i,j\leq m}{\sum} ~\overset{r}{\underset{d=1}{\sum}} \omega_{ij}\alpha_{d_{ij}} y_d=0.  $ By considering the natrual isomorphism
 $ \sigma:L^*/M^*\longrightarrow L $ given by $ \overline{x_i}\mapsto x_i $ and $ \overline{y_s}\mapsto y_s $
 for all $ i,s $ such that $1 \leq s\leq r$  and $1 \leq i\leq m,$ we have $ \sigma (\underset{1\leq i,j\leq m}{\sum}~\overset{r}{\underset{d=1}{\sum}} \omega_{ij}\alpha_{d_{ij}}\overline{y_d}+M^*)=0 $ and so $ \underset{1\leq i,j\leq m}{\sum}~\overset{r}{\underset{d=1}{\sum}} \omega_{ij}\alpha_{d_{ij}}\overline{y_d}\in M^*. $ By the definition of $ M^*, $ we conclude that $ \underset{1\leq i,j\leq m}{\sum}~\overset{r}{\underset{d=1}{\sum}} \omega_{ij}\alpha_{d_{ij}}\overline{y_d}=0.$ It follows that \[ \underset{1\leq i,j\leq m}{\sum} \omega_{ij} e_{ij}+\underset{1\leq i,j\leq m}{\sum}~\overset{r}{\underset{d=1}{\sum}} \omega_{ij}\alpha_{d_{ij}}\overline{y_d}=\underset{1\leq i,j\leq m}{\sum} \omega_{ij} e_{ij}\]\[\in M^*\cap   \langle \overset{r}{\underset{d=1}{\sum}} \alpha_{d_{ij}}\overline{y_d}+e_{ij}| 1 \leq i,j\leq m\rangle.\]  Hence  $ \ker \rho \subseteq M^*\cap   \langle \overset{r}{\underset{d=1}{\sum}} \alpha_{d_{ij}}\overline{y_d}+e_{ij}| 1 \leq i,j\leq m\rangle. $ To prove the opposite inclusion, 
  let $ b $ be an element of $ M^*.$ 
If $ b\in \langle \overset{r}{\underset{d=1}{\sum}} \alpha_{d_{ij}}\overline{y_d}+e_{ij}| 1 \leq i,j\leq m\rangle, $ then
\[ b=\underset{1\leq i,j\leq m}{\sum} \omega_{ij} e_{ij}+\sum_{1\leq i,j\leq m}\sum_{d=1}^{r} \omega_{ij}\alpha_{d_{ij}}\overline{y_d},\] where $  \omega_{ij}$ is a scalar. 
Since $ b\in M^*, $ we have \[  \underset{1\leq i,j\leq m}{\sum}\sum_{d=1}^{r} \omega_{ij}\alpha_{d_{ij}}\overline{y_d}=0\]
and so $  \underset{1\leq i,j\leq m}{\sum} \omega_{ij}\alpha_{d_{ij}}\overline{y_d}=0$ for all $ r $ such that $ 1\leq d\leq r. $ Since
$ \langle \overline{y_1}, \ldots,\overline{y_r}\rangle \cong L^2,  $ we have $ \overline{y_d} $ is not trivial for all $ r $ such that $ 1\leq d\leq r. $  It follows that $ \underset{1\leq i,j\leq m}{\sum} \omega_{ij}\alpha_{d_{ij}}=0.$ Therefore
$ b=\underset{1\leq i,j\leq m}{\sum} \omega_{ij} e_{ij}$ and $  \underset{1\leq i,j\leq m}{\sum} \omega_{ij}\alpha_{d_{ij}}$  $=0.$  
Since
  \begin{align*}
  & \rho(  \sum_{1\leq i,j\leq m} \omega_{ij} e_{ij})=\rho( \sum_{1\leq i,j\leq m} \omega_{ij}  (x_i+L^2\wedge x_j+L^2))=\\&\sum_{1\leq i,j\leq m} \omega_{ij} [x_i,x_j]=\sum_{1\leq i,j\leq m} \sum_{d=1}^{r}\omega_{ij}\alpha_{d_{ij}}y_d=0,
  \end{align*} 
 we get $ \underset{1\leq i,j\leq m}{\sum} \omega_{ij} e_{ij}= \underset{1\leq i,j\leq m}{\sum}\omega_{ij} (x_i+L^2\wedge x_j+L^2)\in \ker \rho $ and so \[ M^*\cap   \langle \sum_{d=1}^{r} \alpha_{d_{ij}}\overline{y_d}+e_{ij}| 1 \leq i,j\leq m\rangle \subseteq \ker \rho. \] It follows that $ M^*\cap   \langle \sum_{d=1}^{r} \alpha_{d_{ij}}\overline{y_d}+e_{ij}| 1 \leq i,j\leq m\rangle=\ker \rho. $ Hence, \[ M^*\cap (L^*)^2 =N\oplus \ker \rho. \] On the other hand, $ M= N\oplus \ker \rho.$ We conclude that $ M=M^*\cap (L^*)^2.$
  The proof is completed.
 \end{proof}
 \section{main results}
 This section is devoted to obtain the Schur multiplier, the tensor square, and the non-abelian exterior square of all $d$-generator generalized Heisenberg Lie algebras of rank $ t $ when $   \frac{1}{2}d(d-1)-3\leq t \leq \frac{1}{2}d(d-1)-1.$ Moreover, we show that such Lie algebras of rank $   \frac{1}{2}d(d-1)-1$ and  $   \frac{1}{2}d(d-1)-2$ are capable. In the final part of the section, these results are extended for nilpotent Lie algebras  of class two with the central factor of dimension $ d $ and derived subalgebra of dimension  $ k$ such that $\frac{1}{2}d(d-1)-3\leq k\leq    \frac{1}{2}d(d-1)-1.$
\\ In the following, we compute the dimension of the image of the homormorphism $\Psi_2$ defined in Theorem  \ref{lkk} for all $d$-generator generalized Heisenberg Lie algebras of rank $ t $ such that $   \frac{1}{2}d(d-1)-3 \leq t\leq   \frac{1}{2}d(d-1)-1.$
 \begin{prop}\label{3}
 Let $ L $ be a $d$-generator generalized Heisenberg Lie algebra of rank $m$ and  $d\geq 3.$ Then 
 \begin{itemize}
 \item[$  (i)$] if $ m= \frac{1}{2}d(d-1)-1 $ or $ m= \frac{1}{2}d(d-1)-2, $ then 
 \[B= \{ \Psi_2(x_n+L^2\otimes x_q +L^2 \otimes x_k +L^2)|1\leq n<q<k \leq d \} \] is a basis of $ \mathrm{Im}\Psi_2 $ and $\dim \mathrm{Im}\Psi_2={\frac{1}{6}d(d-1)(d-2)}.$
 \item[$  (ii)$] If $ m= \frac{1}{2}d(d-1)-3,$ then 
 $\dim \mathrm{Im}\Psi_2={\frac{1}{6}d(d-1)(d-2)}-1$ or  $\dim \mathrm{Im}\Psi_2={\frac{1}{6}d(d-1)(d-2)}.$
 \end{itemize}
\end{prop}
\begin{proof}
\begin{itemize}
 \item[$  (i)$] 
Let $\{ x_1, x_2,\ldots, x_d\}$ be a minimal generating set of $L.$  Clearly, $ L^2=\langle [x_{n},x_{q}]|$  $1\leq n<q\leq d\rangle. $ Since $\dim L^2= \frac{1}{2}d(d-1)-1,$ we have \[ [x_i,x_j]=\sum_{\overset{1\leq n<q\leq d}{(n,q)\neq (i, j) }}\beta_{nq}[x_{n},x_{q}], \] where $  \beta_{nq} $ is a scalar for some  $( i,j) $ such that $1\leq i<j \leq d.$   Let $ \overline{x}=x+L^2. $ We know that $ B_1=\{ [x_r,x_t]|1\leq r<t \leq d,(r,t)\neq (i, j)  \} $ generates $ L^2 $ and $ L/L^2 $ is generated by $ \{\overline{x_1},\ldots,\overline{x_d}\}. $
Since $L^2\cong \underset{\overset{1\leq r<t\leq d}{(r,t)\neq (i, j) }}{\bigoplus} \langle [x_r,x_t]\rangle,$   $ B_1 $ is also linearly independent. Now, since \[L^2\otimes L/L^2\cong \bigoplus_{1\leq k\leq d} \quad \underset{\overset{1\leq r<t\leq d}{(r,t)\neq (i, j) }}{\bigoplus}
\langle [x_r,x_t]\otimes x_k+ L^2\rangle, \] the set $A= \{[x_r,x_t]\otimes \overline{x_k} |1\leq k\leq d,1\leq r<t \leq d,(r,t)\neq (i, j) \} $ is a basis of $ L^2\otimes L/L^2. $
We claim that $ \Psi_2(\overline{x_i} \otimes \overline{x_j} \otimes \overline{x_k} ) $ is non-trivial for all $ i,j,k $ such that $1\leq i<j<k \leq d.$
Assume to the  contrary that  $ \Psi_2(\overline{x_i} \otimes \overline{x_j} \otimes \overline{x_k})=0. $ Then $[x_i,x_j]\otimes \overline{x_k} +[x_j,x_k]\otimes \overline{x_i} +[x_k,x_i]\otimes \overline{x_j} =0$ and so
 $ [x_k,x_i]\otimes \overline{x_j} =-[x_i,x_j]\otimes \overline{x_k } -[x_j,x_k]\otimes \overline{x_i},  $ which is a contradiction, since $[x_k,x_i]\otimes \overline{x_j} \in A.$
 Put $B= \{ \Psi_2(\overline{x_n} \otimes \overline{x_q }\otimes \overline{x_k})|1\leq n<q<k\leq d \}.$ \\Suppose that  $ D=\{ \Psi_2(\overline{x_i} \otimes \overline{x_j }\otimes \overline{x_k})|1\leq k\leq d, k\neq j,k\neq i \}$ and 
$ C=\{ \Psi_2(\overline{x_n} \otimes \overline{x_q }\otimes \overline{x_k})|1\leq n<q<k\leq d,(n,q)\neq (i,j) \}. $
Obviously, $ C\cap D=\varnothing $ and $ B=C\cup D. $
We claim that  $B $  is a basis of $\mathrm{Im}\Psi_2.  $  Clearly, $ B $ generates $\mathrm{Im}\Psi_2.  $  Now, we show that $ B $ is linearly independent.  It is enuogh to see that $ C $ and $ D $ are linearly independent. 
\\ First, let $  \underset{\overset{1\leq n<q<k \leq d}{(n,q)\neq (i,j)}}{\sum} \alpha_{nqk} \Psi_2(\overline{x_n} \otimes \overline{x_q} \otimes \overline{x_k} )=0, $ where $\alpha_{nqk}$ is a scalar.
Then $
\underset{\overset{1\leq n<q<k \leq d}{(n,q)\neq (i,j)}}{\sum} \alpha_{nqk}([x_n,x_q]\otimes \overline{x_k } +[x_q,x_k]\otimes \overline{x_n} +[x_k,x_n]\otimes \overline{x_q} )=0.
$ For all $ n,q,k $ such that $ 1\leq n<q<k \leq d$  and $(n,q)\neq (i,j), $  we have $[x_q,x_k]\otimes \overline{x_n} ,[x_q,x_k]\otimes \overline{x_n},$ and $[x_k,x_n]\otimes \overline{x_q}\in A.  $  
 Since $ A $ is linearly independent, we have $ \alpha_{nqk}=0 $ for all $ n,q,k $ such that $ 1\leq n<q<k \leq d$ and $(n,q)\neq (i,j).$
It implies that $C $ is linearly independent. \\
Now, we show that $ D $ is linearly independent. Let $  \underset{k=1}{\overset{d}{\sum}} \lambda_{ijk} \Psi_2(\overline{x_i} \otimes \overline{x_j} \otimes \overline{x_k} )=0, $ where $\lambda_{ijk}$ is a scalar. Since  $[x_i,x_j]=\underset{\overset{1\leq n<q\leq d}{(n,q)\neq (i, j) }}{\sum}\beta_{nq}[x_{n},x_{q}], $ where $ \beta_{nq} $ is a scalar,  we have 
\begin{align*}
&\sum_{k=1}^d \lambda_{ijk} \big{(}[x_i,x_j]\otimes \overline{x_{k}}+[x_j,x_{k}]\otimes \overline{x_i}  +[x_{k},x_i]\otimes \overline{x_j} \big{)}=\\&
 \Big{(}\sum_{k=1}^d  \big{(} \lambda_{ijk}\underset{\overset{1\leq n<q\leq d}{(n,q)\neq (i, j)}}{\sum}  \beta_{n q}
\big{(}[x_{n},x_{q}]\otimes \overline{x_{k}}\big{)}\big{)}\\&+\sum_{k=1}^d \lambda_{ijk} \Big{(}[x_j,x_{k}]\otimes \overline{x_i}+[x_{k},x_i]\otimes \overline{x_j}\Big{)}=0.
 \end{align*}
 For all $ n,q,k $ such that $ 1\leq n<q \leq d,$  $ 1\leq k\leq d, $ and $(n,q)\neq (i,j), $  we have $[x_j,x_{k}]\otimes \overline{x_i},$  $[x_{k},x_i]\otimes \overline{x_j}$ and $[x_n,x_q]\otimes \overline{x_k}\in A.  $  
 Since $ A $ is linearly independent, we have $ \lambda_{ijk}= \beta_{n q}=0 $ for all $ n,q,k $ such that $ 1\leq n<q \leq d,$  $ 1\leq k\leq d, $ and $(n,q)\neq (i,j).$ 
It implies that $D $ is linearly independent. 
 Thus $ B $ is a basis of $  \mathrm{Im}\Psi_2 $ and so
$ \dim \mathrm{Im}\Psi_2= \frac{1}{6}d(d-1)(d-2).$ Similarly, we can see that the result holds for every $d$-generator generalized Heisenberg Lie algebra of rank $ \frac{1}{2}d(d-1)-2.$ 
\item[$  (ii)$]
 Let $\{ x_1, x_2,\ldots, x_d\}$ be a minimal generating set of $L.$  Clearly, $ L^2=\langle [x_{n},x_{q}]|$  $1\leq n<q\leq d\rangle. $ Since $\dim L^2= \frac{1}{2}d(d-1)-3,$ we have \[ [x_i,x_j]=\underset{\overset{1\leq n<q\leq d}{(n,q)\notin\{ (i, j), (t, s),(t', s') \} }}{\sum}\beta_{nq}[x_{n},x_{q}], \]\[ [x_t,x_s]=\underset{\overset{1\leq n_1<q_1\leq d}{(n_1,q_1)\notin\{ (i, j), (t, s),(t', s') \} }}{\sum}\beta_{n_1q_1}[x_{n_1},x_{q_1}], \] and $ [x_{t'},x_{s'}]=\underset{\overset{1\leq n_2<q_2\leq d}{(n_2,q_2)\notin\{ (i, j), (t, s),(t', s') \}}}{\sum}\beta_{n_2q_2}[x_{n_2},x_{q_2}], $  where $  \beta_{n q}, $ $  \beta_{n_1q_1} $ and  $  \beta_{n_2q_2} $ are scalars for some  $ (i,j),(t,s),(t',s') $ such that $1\leq t'<s' \leq d,$  $1\leq t<s \leq d$ and $1\leq i<j\leq d.$ If $ \Psi_2(\overline{x_i} \otimes \overline{x_j} \otimes \overline{x_k} ) $ is trivial for some $ i,j,k $ such that $1\leq i<j<k \leq d,$ then  $[x_i,x_j]=[x_j,x_k]=[x_k,x_i] =0.$ Otherwise, $ \Psi_2(\overline{x_i} \otimes \overline{x_j} \otimes \overline{x_k} ) $ is non-trivial. By a similar way used in  part $ (i), $ we have 
 $\dim \mathrm{Im}\Psi_2={\frac{1}{6}d(d-1)(d-2)}-1$ or  $\dim \mathrm{Im}\Psi_2={\frac{1}{6}d(d-1)(d-2)}.$ The proof is completed.
  \end{itemize}
\end{proof}
We are ready to determine the structure of the Schur multiplier of a $d$-generator generalized Heisenberg Lie algebra of  rank $ m $ such that $   \frac{1}{2}d(d-1)-3 \leq m\leq   \frac{1}{2}d(d-1)-1.$
\begin{thm}\label{k5}
 Let $ L $ be a $d$-generator generalized Heisenberg Lie algebra of rank $ m$ and  $d\geq 3.$ 
 Then
 \begin{itemize}
 \item[$ (i) $]If $ m=\frac{1}{2}d(d-1)-1, $ then $\mathcal{M}(L) \cong A(\frac{1}{3}d(d-1)(d+1)-d+1)). $
 \item[$ (ii) $]If $ m=\frac{1}{2}d(d-1)-2, $ then $\mathcal{M}(L) \cong A(\frac{1}{3}d(d-1)(d+1)-2d+2)). $
  \item[$ (iii) $]If $ m=\frac{1}{2}d(d-1)-3, $ then $\mathcal{M}(L) \cong A(\frac{1}{3}d(d-1)(d+1)-3d+3))$ or
   $\mathcal{M}(L) \cong A(\frac{1}{3}d(d-1)(d+1)-3d+2)).$
 \end{itemize}
 \end{thm}
\begin{proof}
\begin{itemize}
 \item[$ (i) $]
 Theorem \ref{lkk}, Lemma \ref{1} and Theorem \ref{2} imply that 
\[\dim  \mathcal{M}(L)=\dim (L^2\otimes L/L^2)-\dim \mathrm{Im}\Psi_2 +\dim (L/L^2\wedge L/L^2)-\dim L^2.\] By \cite[Lemma 2.6 ]{nin} and Proposition \ref{3}$ (i),$ we have $ \dim \mathcal{M}(L/L^2)= \frac{1}{2}d(d-1)$ and $\dim \mathrm{Im}\Psi_2 =\frac{1}{6}d(d-1)(d-2).$ Now $ \dim  (L^2\otimes L/L^2)=\frac{1}{2}d^2(d-1) -d.$ Thus 
\begin{align*}
&\dim  \mathcal{M}(L)=\\&\frac{1}{2}d^2(d-1)-d -\frac{1}{6}d(d-1)(d-2)+\frac{1}{2}d(d-1) - \frac{1}{2}d(d-1)+1=\\&
\frac{1}{2}d^2(d-1) -\frac{1}{6}d(d-1)(d-2)-d+1=\\&
\frac{1}{3}d(d-1)(d+1)-d+1.
\end{align*}
Therefore $ \dim  \mathcal{M}(L) =  \frac{1}{3}d(d-1)(d+1)-d+1.$ It completes the proof. 
 \item[$ (ii) $]
  Theorem \ref{lkk}, Lemma \ref{1} and Theorem \ref{2} imply that
\[\dim  \mathcal{M}(L)=\dim (L^2\otimes L/L^2)-\dim \mathrm{Im}\Psi_2 +\dim (L/L^2\wedge L/L^2)-\dim L^2.\]
 Using Proposition \ref{3}$ (i)$ and \cite[Lemma 2.6 ]{nin}, we have \[\dim \mathrm{Im}\Psi_2 =\frac{1}{6}d(d-1)(d-2)~\text{and}~ \dim \mathcal{M}(L/L^2)= \frac{1}{2}d(d-1).\] Now, $ \dim  (L^2\otimes L/L^2)=\frac{1}{2}d^2(d-1) -2d.$ Therefore
\begin{align*}
&\dim  \mathcal{M}(L)=\\&\frac{1}{2}d^2(d-1)-2d+\frac{1}{2}d(d-1) -\frac{1}{6}d(d-1)(d-2)- \frac{1}{2}d(d-1)+2=\\&
\frac{1}{2}d^2(d-1) -\frac{1}{6}d(d-1)(d-2)-2d+2=\\&
\frac{1}{3}d(d-1)(d+1)-2d+2.
\end{align*}
Hence $ \dim  \mathcal{M}(L) =  \frac{1}{3}d(d-1)(d+1)-2d+2.$ It completes the proof. 
 \item[$ (iii) $]Similarly, the result holds for every $d$-generator generalized Heisenberg Lie algebra of rank $ \frac{1}{2}d(d-1)-3.$ 
 \end{itemize}
\end{proof}
Now we show that the converse of Theorem \ref{k5}$ (i) $ and $ (ii) $ is also true.
\begin{thm}
Let $ L $ be a $d$-generator generalized Heisenberg Lie algebra of rank $ m.$ Then
\begin{itemize}
\item[$ (i) $]$m= \frac{1}{2}d(d-1)-1$ if and only if $ \dim  \mathcal{M}(L) =  \frac{1}{3}d(d-1)(d+1)-d+1.$
\item[$ (ii) $]$m= \frac{1}{2}d(d-1)-2$ if and only if $ \dim  \mathcal{M}(L) =  \frac{1}{3}d(d-1)(d+1)-2d+2.$
\end{itemize}
\end{thm}
\begin{proof}
Let $ \dim  \mathcal{M}(L) =  \frac{1}{3}d(d-1)(d+1)-d+1.$  By a similar way  used in the proof of Theorem  \ref{k5}, we can
see that if $\dim L^2 < \frac{1}{2}d(d-1)-1,$ then  $ \dim  \mathcal{M}(L) < \frac{1}{3}d(d-1)(d+1)-d+1.$ Thus $\dim L^2 = \frac{1}{2}d(d-1)-1.$ The converse holds by Theorem  \ref{k5}. Let $ \dim  \mathcal{M}(L) =  \frac{1}{3}d(d-1)(d+1)-2d+2.$  By a similar way  used in the proof of Theorem  \ref{k5}, we can
see that if $\dim L^2 < \frac{1}{2}d(d-1)-2,$ then  $ \dim  \mathcal{M}(L) < \frac{1}{3}d(d-1)(d+1)-2d+2.$ Thus $\dim L^2 = \frac{1}{2}d(d-1)-2.$ The converse holds by Theorem  \ref{k5}.
\end{proof}
We are a position to determine the capability of all $d$-generator generalized Heisenberg Lie algebras of rank $ \frac{1}{2}d(d-1)-1$ or $ \frac{1}{2}d(d-1)-2$. 

\begin{thm}\label{ll9}
Let $ L $ be a $d$-generator generalized Heisenberg Lie algebra of rank $ \frac{1}{2}d(d-1)-1$ or $ \frac{1}{2}d(d-1)-2$. 
Then $ L $ is capable.
\end{thm}
\begin{proof}
Theorem \ref{k5}  implies that  $ \dim  \mathcal{M}(L)>\dim  \mathcal{M}(L/K)$ for every one-dimensional central ideal $K $ of $L  $ and so the homomorphism $ \mathcal{M}(L)\rightarrow \mathcal{M}(L/K) $ is not monomorphism.
Thus the result follows from \cite[Corollary 4.6]{sa}. 
\end{proof}

From \cite{Mon}, a Lie algebra $ S $ is called stem if $ Z(S) \subseteq S^2.$ The following lemma is useful to prove Theorem
\ref{cover}.
\begin{lem}\label{f}
Let $ L $ be a $d$-generator generalized Heisenberg Lie algebra of rank $m$  such that  $   \frac{1}{2}d(d-1)-3 \leq m\leq   \frac{1}{2}d(d-1)-1$ and  $ L^* $ be a cover for $ L. $ Then $L^*$ is a stem nilpotent Lie algebra of class $3.$ 
\end{lem}
\begin{proof}
Let $L^*$ be a cover Lie algebra of  $L.$ Then there exists an ideal $B$ of $L^*$
such that $B\cong  \mathcal{M}(L),$ $ B \subseteq Z(L^*) \cap (L^*)^2,$ and $ L^*/B  \cong L.$ Since $Z(L^*/B) =
(L^*/B)^2 = (L^*)^2/B,$ we have $Z(L^*) \subseteq (L^*)^2.$ We claim that $L^*  $ is nilpotent of
class $3.$ Clearly $L^*$ is not abelian since $L$ is non-abelian. Assume to the  contrary that  $L^*$ is
nilpotent of class $ 2.$ Since $B \subseteq (L^*)^2 = Z(L^*),$ we have 
\begin{align*}
&d = \dim L/Z(L) = \dim L/L^2 =
\dim \big{(}(L^*/B)/(L^*/B)^2\big{)} \\&= \dim(L^*/(L^*)^2) = \dim(L^*/Z(L^*)).\end{align*}
By \cite[Lemma 14]{Mon}, we have $\dim (L^*)^2 \leq  \frac{1}{2}d(d-1). $ Now we have\[
\dim  \mathcal{M}(L)=\dim B  \leq  \dim (L^*)^2 \leq \frac{1}{2}d(d-1).\] By Theorem \ref{k5},  we have a contradiction. Thus $ L^* $  is nilpotent of class $3.$ The result holds.
\end{proof}
We are ready to discribe a cover Lie algebra $ L^* $ of a $d$-generator generalized Heisenberg Lie algebra $ L $ of rank $ m $ when $ \frac{1}{2}d(d-1)-3\leq m\leq \frac{1}{2}d(d-1)-1$.
\begin{thm}\label{cover}
Let $ L $ be a $d$-generator generalized Heisenberg Lie algebra of rank $m.$ Then  
\begin{itemize}
\item[$(i)  $] if $ m=\frac{1}{2}d(d-1)-1, $ then
$L^*$ is the cover of $ L $  if and only if $L^*$ is a stem nilpotent Lie algebra of class $3$  with a central ideal $ B $ such that $ L^*/B\cong L $ and $B=(L^*)^3 \cong \mathcal{M}(L)$ or $B= (L^*)^3 \oplus A(1) \cong \mathcal{M}(L).$
\item[$(ii)  $]If $ m=\frac{1}{2}d(d-1)-2, $ then
$L^*$ is the cover of $ L $ if and only if $L^*$ is a stem nilpotent Lie algebra of class $3$  with a central ideal $ B $ such that $ L^*/B\cong L $ and $B=
(L^*)^3 \cong \mathcal{M}(L);$ or  $
B=(L^*)^3 \oplus A(1) \cong \mathcal{M}(L);$ or $B=
 (L^*)^3 \oplus A(2) \cong \mathcal{M}(L).$
 \item[$(iii)  $]If $ m=\frac{1}{2}d(d-1)-3, $ then
$L^*$ is the cover of $ L $ if and only if $L^*$ is a stem nilpotent Lie algebra of class $3$  with a central ideal $ B $ such that $ L^*/B\cong L $ and $B=(L^*)^3 \cong \mathcal{M}(L);$ or
$B=(L^*)^3\oplus A(1) \cong \mathcal{M}(L);$ or
 $
 B=(L^*)^3 \oplus A(2) \cong \mathcal{M}(L); $  or  $
 B=(L^*)^3\oplus A(3) \cong \mathcal{M}(L).$ 
\end{itemize}
\end{thm}
\begin{proof}
\begin{itemize}
\item[$(i)  $] 
Let $L^*$ be a cover Lie algebra of  $L.$ Then there exists an ideal $B$ of $L^*$
such that $B\cong  \mathcal{M}(L),$ $ B \subseteq Z(L^*) \cap (L^*)^2,$ and $ L^*/B  \cong L.$  Since $
L^*/(L^*)^3  $ is of class two and $\dim \big{(}(L^*/(L^*)^3)/(L^*/(L^*)^3)^2\big{)} = d,$ we have $\dim((L^*)^2/(L^*)^3) \leq
\frac{1}{2}d(d-1),$ by  \cite[Lemma 14]{Mon}. Using Lemma \ref{f}, $L^*$ is a stem nilpotent Lie algebra of class $3$  and so $(L^*)^3 \subseteq B $ and $\dim (L^*)^2 -\dim B = \dim L^2.$ Therefore 
\begin{align*}
&\frac{1}{2}d(d-1)-1+ \dim B - \dim(L^*)^3=\dim L^2 + \dim B - \dim (L^*)^3 = \\&
\dim (L^*)^2 - \dim (L^*)^3 = \dim((L^*)^2/(L^*)^3) \leq \frac{1}{2}d(d-1).
\end{align*}
Hence $  \frac{1}{2}d(d-1)-1+ \dim B - \dim (L^*)^3\leq \frac{1}{2}d(d-1) $ and so\[ 0\leq \dim B - \dim(L^*)^3 \leq 1. \] It follows that $ B\cong (L^*)^3\oplus A(1)$ or $ B= (L^*)^3. $ The proof is completed.
  \item[$(ii)  $]
   Let $L^*$ be a cover Lie algebra of $L.$ Then there exists an ideal $B$ of $L^*$
such that $B\cong  \mathcal{M}(L),$ $ B \subseteq Z(L^*) \cap (L^*)^2,$ and $ L^*/B  \cong L.$ Since $
L^*/(L^*)^3  $ is of class two and $\dim \big{(}(L^*/(L^*)^3)/(L^*/(L^*)^3)^2 \big{)}= d,$ we have \[\dim((L^*)^2/(L^*)^3) \leq
\frac{1}{2}d(d-1),\] by  \cite[Lemma 14]{Mon}. Using Lemma \ref{f}, $L^*$ is a stem nilpotent Lie algebra of class $3$  and so $(L^*)^3 \subseteq B $ and $\dim (L^*)^2 -\dim B = \dim L^2$  and so $(L^*)^3 \subseteq B $ and $\dim (L^*)^2 -\dim B = \dim L^2.$ Therefore 
\begin{align*}
&\frac{1}{2}d(d-1)-2+ \dim B - \dim(L^*)^3=\dim L^2 + \dim B - \dim (L^*)^3 = \\&
\dim (L^*)^2 - \dim (L^*)^3 = \dim((L^*)^2/(L^*)^3) \leq \frac{1}{2}d(d-1).
\end{align*}
Hence $  \frac{1}{2}d(d-1)-2+ \dim B - \dim(L^*)^3\leq \frac{1}{2}d(d-1) $ and so 
\[0\leq \dim B - \dim(L^*)^3 \leq 2. \] Thus $ B\cong (L^*)^3\oplus A(2),$ $ B\cong (L^*)^3\oplus A(1),$ or
$ B= (L^*)^3.$
The proof is completed.
\item[$(iii)  $]Similarly, the result holds for every $d$-generator generalized Heisenberg Lie algebra of rank $ \frac{1}{2}d(d-1)-3.$ 
\end{itemize}
\end{proof}
 From \cite{el},  $L\wedge L,  $ $ L\otimes L, $ and  $ J_2(L) $ are used to denote the exterior square, the non-abelian tensor square, and  the kernel of the commutator map $ \kappa: L\otimes L\rightarrow L^2$ of
a Lie algebra $L,$ respectively. The authors assume that the reader is familiar with
these concepts.  The following results give the structures of the non-abelian tensor square and the exterior square  of a  $d$-generator generalized Heisenberg Lie algebra of rank $ m $ such that $   \frac{1}{2}d(d-1)-3 \leq m\leq   \frac{1}{2}d(d-1)-1.$
 \begin{thm}\label{lklk9}
 Let $ L $ be a $d$-generator generalized Heisenberg Lie algebra of rank $m.$  Then
 \begin{itemize}
\item[$(i)  $] if $ m=\frac{1}{2}d(d-1)-1, $ then
  \begin{align*} &L\wedge L\cong A(\frac{1}{6}(d-1)(2d^2+5d-6)-1), \\
 &L\otimes L\cong  A(\frac{1}{3}d(d^2+3d-4)), \\
 & J_2(L)\cong  A((\frac{1}{6}d(2d^2+3d-5)+1).
 \end{align*}
 \item[$(ii)  $] If $ m=\frac{1}{2}d(d-1)-2, $ then
  \begin{align*} &L\wedge L\cong A(\frac{1}{6}(d-1)(2d^2+5d-12)-2), \\
 &L\otimes L\cong  A(\frac{1}{3}d(d^2+3d-7)), \\
 & J_2(L)\cong  A(\frac{1}{6}d(2d^2+3d-11)+2). 
 \end{align*}
  \end{itemize}
 \end{thm}
 \begin{proof}
 \begin{itemize}
\item[$(i)  $] By \cite[Theorem 35$(iii)$]{el} and Theorem \ref{k5}, we have \[  L\wedge L\cong \mathcal{M}(L)\oplus L^2\cong A(\frac{1}{6}(d-1)(2d^2+5d-6)-1)).\] Using \cite[Lemmas 2.2 and 2.3]{ni223}, we have $L/L^2\square L/L^2 \cong  A( \frac{1}{2}d(d+1)).$  Now \cite[Theorem 2.5]{ni223} shows that \begin{align*} &L\otimes L\cong  (L\wedge L) \oplus (L/L^2\square L/L^2 )\cong\\& A(\frac{1}{6}(d-1)(2d^2+5d-6)-1))\oplus A( \frac{1}{2}d(d+1))\cong  \\&A(\frac{1}{3}d(d^2+3d-4)).\end{align*} Now \cite[Corollary 2.6]{ni223} implies that
$J_2(L)\cong  A((\frac{1}{6}d(2d^2+3d-5)+1).$ The proof is completed.
\item[$(ii)  $]   By a similar way used in the proof of part $ (i), $ we can obtain the result.
 \end{itemize}
 \end{proof}
 \begin{thm}
 Let $ L $ be a $d$-generator generalized Heisenberg Lie algebra of rank $ \frac{1}{2}d(d-1)-3.$ Then
 \begin{itemize}
\item[$(i)  $] if $\mathcal{M}(L) \cong A(\frac{1}{3}d(d-1)(d+1)-3d+3)),$ then
 \begin{align*} &L\wedge L\cong A(\frac{1}{6}(d-1)(2d^2+5d-18)-3), \\
 & L\otimes L\cong  A(\frac{1}{3}d(d^2+3d-10)),\\
  &J_2(L)\cong  A(\frac{1}{6}(d+1)(2d^2+3d-20)+3).
\end{align*}
\item[$(ii)  $] If $\mathcal{M}(L) \cong A(\frac{1}{3}d(d-1)(d+1)-3d+2)),$ then
\begin{align*}  
  &L\wedge L\cong A(\frac{1}{6}(d-1)(2d^2+5d-18)-4),
 \\& L\otimes L\cong  A(\frac{1}{3}d(d^2+3d-10)-1),
\\& J_2(L)\cong  A(\frac{1}{6}(d+1)(2d^2+3d-20)+2). 
 \end{align*}
    \end{itemize}
 \end{thm}
 \begin{proof}
The result is obtained by a similar way used in the proof of Theorem \ref{lklk9}$(i).$ 
 \end{proof}
The following result gives the structures of the non-abelian tensor square,  the exterior square, and the Schur multplier  of a nilpotent Lie algebra $ L $ of class two such that $\dim(L/Z(L)) = d$ and  $   \frac{1}{2}d(d-1)-3 \leq \dim L^2\leq   \frac{1}{2}d(d-1)-1.$
 \begin{thm}
 Let $ L$ be an $n$-dimensional Lie algebra of class two such that $\dim(L/Z(L)) = d.$ Then
 \begin{itemize}
\item[$(i)  $] if $\dim L^2 = \frac{1}{2}d(d-1)-1, $ then
\begin{align*}
&\mathcal{M}(L)\cong  A(\frac{1}{3}d(d-1)(d+1)+\frac{1}{2}(t-1)(t+2d)+1)).\\
  &L\wedge L\cong A(\frac{1}{6}(d-1)(2d^2+5d-6)+\frac{1}{2}t(t-1)+dt-1)).\\
  &L\otimes L\cong A(\frac{1}{3}d(d^2+3d-4)+\frac{1}{2}(2t^2+4dt+d^2+d)).\\
  &J_2(L)\cong A(\frac{1}{3}d(d-1)(d+1)+\frac{1}{2}(2t^2+4dt+d^2-d)+1).
  \end{align*}
  \item[$(ii)  $] If $\dim L^2 = \frac{1}{2}d(d-1)-2, $ then
\begin{align*}
&\mathcal{M}(L)\cong  A(\frac{1}{3}d(d-1)(d+1)+\frac{1}{2}(t-1)(t+2d)-d+2)).\\
  &L\wedge L\cong A(\frac{1}{6}(d-1)(2d^2+5d-12)+\frac{1}{2}t(t-1)+dt-2)).\\
  &L\otimes L\cong A(\frac{1}{3}d(d^2+3d-7)+\frac{1}{2}(2t^2+4dt+d^2+d)).\\
  &J_2(L)\cong A(\frac{1}{3}d(d-1)(d+1)+\frac{1}{2}(2t^2+4dt+d^2-3d)+2).
  \end{align*}
  \item[$(iii)  $] If $\dim L^2 = \frac{1}{2}d(d-1)-3, $ then
\begin{align*}
&\mathcal{M}(L)\cong  A(\frac{1}{3}d(d-1)(d+1)+\frac{1}{2}(t-1)(t+2d)-2d+3).\\
  &L\wedge L\cong A(\frac{1}{6}(d-1)(2d^2+5d-18)+\frac{1}{2}t(t-1)+dt-3).\\
  &L\otimes L\cong A(\frac{1}{3}d(d^2+3d-10)+\frac{1}{2}(2t^2+4dt+d^2+d)).\\
  &J_2(L)\cong A(\frac{1}{3}d(d-1)(d+1)+\frac{1}{2}(2t^2+4dt+d^2-5d)+3).
  \end{align*}
  or
  \begin{align*}
&\mathcal{M}(L)\cong  A(\frac{1}{3}d(d-1)(d+1)+\frac{1}{2}(t-1)(t+2d)-2d+2).\\
  &L\wedge L\cong A(\frac{1}{6}(d-1)(2d^2+5d-18)+\frac{1}{2}t(t-1)+dt-4).\\
  &L\otimes L\cong A(\frac{1}{3}d(d^2+3d-10)+\frac{1}{2}(2t^2+4dt+d^2+d)-1).\\
  &J_2(L)\cong A(\frac{1}{3}d(d-1)(d+1)+\frac{1}{2}(2t^2+4dt+d^2-5d)+2).
  \end{align*}
  \end{itemize}
 \end{thm}
 \begin{proof}
 \begin{itemize}
\item[$(i)  $] 
 By  \cite[Proposition 2.2]{ni3j}, we have $ L\cong H\oplus A(t) $ such that  $ H $ be a $d$-generator generalized Heisenberg Lie algebra of rank $\frac{1}{2}d(d-1)-1$ and $ A(t) $ is an abelian Lie algebra. Using Theorem \ref{lklk9} and \cite[Proposition 10]{el},
 we have 
  \begin{align*}
  &L\wedge L\cong (H\wedge H)\oplus (A(t)\wedge A(t)) \oplus (H/H^2\otimes A(t))\cong\\&
   A(\frac{1}{6}(d-1)(2d^2+5d-6)-1))\oplus A(\frac{1}{2}t(t-1))\oplus A(dt)\cong\\& A(\frac{1}{6}(d-1)(2d^2+5d-6)+\frac{1}{2}t(t-1)+dt-1).
  \end{align*}
  Hence $  \mathcal{M}(L)\cong   A(\frac{1}{3}d(d-1)(d+1)+\frac{1}{2}t(t-1)+dt-d+1)),$ by  
  \cite[Theorem 35$(iii)$]{el}.
  Also  \cite[Theorem 2.5 and Corollary 2.6]{ni223} imply that
  \begin{align*}
  &L\otimes L\cong A(\frac{1}{3}d(d^2+3d-4)+\frac{1}{2}(2t^2+4dt+d^2+d)),\\
  &J_2(L)\cong A(\frac{1}{3}d(d-1)(d+1)+\frac{1}{2}(2t^2+4dt+d^2-d)+1),
  \end{align*}
  as required.
  \item[$(ii),(iii)  $]By a similar technique used in the proof of part $ (i), $ we can obtain the result.
  \end{itemize}
 \end{proof}
 In the following,
  we determine the capability of a nilpotent Lie algebra $L$ of class two such that $\dim(L/Z(L)) =d$ and $ \dim L^2= \frac{1}{2}d(d-1)-1$ or $ \dim L^2= \frac{1}{2}d(d-1)-2.$ 
\begin{cor}
 Let $ L$ be an $n$-dimensional Lie algebra of class two such that $dim(L/Z(L)) = d$ and $\dim L^2 = \frac{1}{2}d(d-1)-1$ or
 $\dim L^2= \frac{1}{2}d(d-1)-2.$
 Then $L$ is capable. 
 \end{cor}
 \begin{proof}
The result follows from \cite[Proposition 2.2]{ni3j}  and Theorem \ref{ll9}.
 \end{proof}
 \section*{Acknowledgment}
Farangis Johari was  supported  by postdoctoral grant  `` CAPES/PRINT-Edital $n^\circ 41/2017$, Process number:88887.511112/2020-00, ''  at  Federal University of Minas Gerais.
 
\end{document}